\documentclass[12pt,letterpaper]{article}
\usepackage{pslatex}
\usepackage{fancyhdr}
\usepackage{graphicx}
\usepackage{geometry}
\RequirePackage[latin1]{inputenc}
 \RequirePackage[T1]{fontenc}

\def\figurename{Figure} % Replace the colon that normally appears after the Figure number by a period.
\makeatletter
\renewcommand{\fnum@figure}[1]{\figurename~\thefigure.}
\makeatother

\def\tablename{Table} % Replace the colon that normally appears after the Figure number by a period.
\makeatletter
\renewcommand{\fnum@table}[1]{\tablename~\thetable.}
\makeatother

\usepackage{amsmath}
\usepackage{amssymb}
\usepackage{amsfonts}
\usepackage{amsthm,amscd}

\newtheorem{theorem}{Theorem}[section]

\newtheorem{proposition}[theorem]{Proposition}
\newtheorem{remark}[theorem]{Remark}
\newtheorem{Definition}[theorem]{Definition}

\numberwithin{equation}{section}

\def\P{\mathbb P}

\def\R{\mathbb R}
\def\E{\mathbb E}

\def\Q{\mathbb Q}

\def\E{\mathbb E}

\title{Backward stochastic differential equations with time-delayed generators and integrable parameters}

\author{Auguste Aman$^{\, b,}$  \footnote{Corresponding author: augusteaman5@yahoo.fr}\; and \; Yong Ren$^{\,a,}$\footnote{renyong@126.com}\\\\
a. Université Félix H. Boigny, UFR Mathématiques et Informatique,\\  Abidjan, Côte d'Ivoire\\
b. Anhui Normal University, Department of Mathematics, Wuhu, Chine}

\begin{document}
\maketitle
\date{}
\begin{abstract}
In this note, we derive an existence and uniqueness results for delayed backward stochastic differential equation with only integrable data. 
\end{abstract}

\vspace{.08in} \noindent {\bf 2000 MR Subject Classification:} 60H15; 60H20; 60H30\\
\vspace{.08in} \noindent \textbf{Keywords:} Backward stochastic differential equation; time delayed generator; $L^{1}$-solutions

\section{Introduction}
In this note, we study the problem of existence and uniqueness result for the backward stochastic differential equations driven by a standard Brownian motion $W$ and with delayed generator $f$, terminal value $\xi$ and fixed terminal time $T$. More precisely we have 
\begin{eqnarray}\label{eq1}
Y(t)=\xi +\int_{t}^{T}f(s,Y_{s},Z_{s})ds-\int_{t}^{T}Z(s)dW(s), 0\leq t\leq T,
\end{eqnarray}
where $(Y_s,Z_{s})=( Y(s+u),Z(s+u) )_{-T\leq u \leq 0}$. Introduced by Delong and Imkeller in \cite{Imk1} and with added Poisson random measures in \cite{Imk2}, this type of BSDEs gave a Mathematic formulation of many problem in finance and insurance, (see \cite{Del}). Since this three results, a lot of works concerned the delayed BSDEs have been done in literature. Among others one can cite these works of Ren et al., \cite{Ral}, Coulibaly and Aman,\cite{CA}, Tuo et al., \cite{Tal}, Dos Reis et al., \cite{Reis} and Ren et al.\cite{ROA}. In all this works, existence and uniqueness result are provided under delayed Lipschitz or non-Lipschitz condition with $p$-integrable data, for $p> 1$. For the BSDEs without delay, inspired by the notion of $g$-martingale, introduced by Peng \cite{Peng}, Briand et al.\cite{Bral} establish an existence and uniqueness result with only integrable data. This type of problem has never been tackled in the scientific literature of the BSDEs with delay. It is very important in our opinion that this gap can be corrected. Indeed, knowing that all the models described by the BSDEs without delay have their version with delay, this generalization will allow to extend the study of existing models to those with delay. For example, let us consider the financial market which consists of two tradeable instruments. A risk-free asset whose price $B = (B(t), 0\leq t\leq T)$ is governed by
\begin{eqnarray*}
	\frac{dB(t)}{B(t)}=r(t)dt,\;\;\; B(0)=1,
\end{eqnarray*}
where $r = (r(t), 0\leq t\leq T)$ denotes the risk-free interest rate in the market. A risky bond (action) whose price is denoted by  $D = (D(t), 0 \leq t \leq T)$ and  described by
\begin{eqnarray*}
	\frac{dD(t)}{D(t)}=[r(t)+\sigma(t)\theta(t)]dt+\sigma(t)dW(t),\;\;\; D(0)=d_0,
\end{eqnarray*}
where $\sigma$ and $\theta$ are respectively the volatility and premium risky of the action. Let denote by $(X(t),\; 0\leq t\leq T)$ and $(\pi(t),\; 0\leq t\leq T)$ respectively an investment portfolio and an amount invested in bond. We have
\begin{eqnarray*}
dX(t)=\pi(t)[\mu(t)dt+\sigma(t)dW]+(X(t)-\pi(t))r(t)dt,\;\;\; X(0)=x,
\end{eqnarray*}
where $\mu(t)=r(t)+\sigma(t)\theta(t)$. Setting 
\begin{eqnarray*}
Y(t)=X(t)\exp\left(-\int_0^tr(s)ds\right),\;\; Z(t)=\pi(t)\sigma(t)\exp\left(-\int_0^tr(s)ds\right),
\end{eqnarray*}
it follows from Girsanov theorem that there exist a probability measure $\Q$ equivalent to $\P$ and a $\Q$-brownian motion $W^{\Q}$ such that $(Y,Z)$ satisfy
\begin{eqnarray*}
dY(t)=Z(t)dW^{\Q}(t),\;\;\; Y(0)=x.
\end{eqnarray*}
If we want to use such investment strategy $Z$ and an investment portfolio $Y$ to replicate a liability or meet a target of the form $Y(0)+(\xi-Y(0))^{+}$ (which appear on the option-based portfolio assurance), $(Y,Z)$ satisfies the following time-delayed BSDE
\begin{eqnarray}
Y(t)=Y(0)+(\xi-Y(0))^{+}-\int^T_tZdW^{\Q}(t).\label{Delay}
\end{eqnarray}
The question of solving BSDE \eqref{Delay} when the data
are only integrable comes natural.

The aim of this note is to answer this question and provide an existence and uniqueness result for delayed BSDEs when the data are only integrable.

 It is organized as follows. In Section 2, we give all notations and the setting of the problem. The Section 3 is devoted to prove existence and uniqueness result.

\section{Notations and Setting of the Problem}
Throughout this paper, we consider the probability space $(\Omega,\mathcal{F},\mathbb{P})$ and $T$ a real and positive constant. We define on the above probability space, $W=(\displaystyle W_{t})_{0\leq t\leq T}$ a standard Brownian motion in values in $ \mathbb{R}^{d}$. Let denote $\mathcal{N}$ the class of $\P$-null sets of $\mathcal{F}$ and define
\begin{eqnarray*}
	\mathcal{F}_t=\mathcal{F}^{W}_t\vee\mathcal{N},
\end{eqnarray*}
 where $\mathcal{F}_{t}^{W}=\sigma(W_s,\; 0\leq s\leq t)$. The standard inner product of $\mathbb{R}^{k}$ is denoted by
$\langle \cdot,\cdot\rangle$ and the Euclidean norm by $|\cdot|$. A
norm on $\mathbb{R}^{d\times k}$ is defined by
$\sqrt{Tr(zz^{\star})}$, where $z^{\star}$ is the transpose of
$z$ will also be denoted by $|\cdot|$.

Next for any real $p>0$ , let consider these following needed spaces.

\begin{description}
\item $\bullet$ Let $\mathcal{S}^{p}(\mathbb{R}^{k})$ denote the set of $ \mathbb{R}^{k}$-valued, adapted and càdlàg processes $(\varphi(t))_{t\geq 0}$ such that
\begin{eqnarray*}
\|\varphi(t)\|_{\mathcal{S}^{p}(\mathbb{R}^{k})}=\left[\mathbb{E}\left(\sup_{0\leq t\leq T}|\varphi(t)|^{p}\right)\right]^{1\wedge 1/p}< + \infty.
\end{eqnarray*}
\item $\bullet$ Let $\mathcal{H}^{p}(\mathbb{R}^{k \times d})$ denote the set of predictable processes $(\psi(t))_{t\geq 0}$ with value in $\mathbb{R}^{k \times d}$ such that
\begin{eqnarray*}
\|\psi\|_{\mathcal{H}^{p}(\mathbb{R}^{k\times d})}=\left[\mathbb{E}\left(\int_{0}^{T}|\psi(s)|^{2}ds\right)^{p/2}\right]^{1\wedge  1/p}.
\end{eqnarray*}
\item $\bullet$ Let $(D)$ be the space of progressively continuous processes $\varphi$ such that the family $\left\{\varphi(\tau),\; \tau\in \sum_T\right\}$ is uniformly integrable. For a process $\varphi$ in $(D)$, we put
\begin{eqnarray*}
\|\varphi\|_{1}=\sup\left\{\E(Y_{\tau}),\;\; \tau\in \sum\,_T\right\}.
\end{eqnarray*}
$(D)$ is complete space under this norm.
\end{description}
\begin{remark}
\begin{itemize}
	\item [(i)] If $p\geq 1,\; \|.\|_{\mathcal{S}^{p}(\mathbb{R}^{k})}$ is a norm on $\mathcal{S}^p(\R^{k})$ and if  $p\in (0,1),\; (X,X')\mapsto \|X-X'\|_{\mathcal{S}^{p}(\mathbb{R}^{k})}$ is a metric on $\mathcal{S}^p(\R^{k})$.
	\item [(ii)] For $p\geq 1$, $\mathcal{H}^p(\R^{k\times d})$ is a banach space with norm $\|.\|_{\mathcal{H}^{p}(\mathbb{R}^{k\times d})}$ and for $p\in (0,1),\, \mathcal{H}^{p}(\mathbb{R}^{k\times d})$ is a complete metric space with the resulting distance.
\end{itemize}

\end{remark}
We also consider these two additive spaces
\begin{description}
\item $\bullet$ Let $L_{-T}^{2}(\mathbb{R}^{k\times d})$ denote the space of jointly measurable functions         $z:[-T,0]\rightarrow\mathbb{R}^{k \times d}$ such that
\begin{eqnarray*}
\int_{-T}^{0}|z(t)|^{2}dt<+\infty.
\end{eqnarray*}
\item $\bullet$ Let $L_{-T}^{\infty}(\mathbb{R}^{k})$ denote the space of bounded and  jointly measurable functions $y:[-T,0]\rightarrow\mathbb{R}^{k}$ such that
\begin{eqnarray*}
\sup_{-T\leq t\leq 0}|y(t)|^{2}<+\infty.
\end{eqnarray*}
\end{description}
We now recall our backward stochastic differential equation with time delayed generator
\begin{eqnarray}\label{A}
Y(t)=\xi+\int_t^T f(s,Y_s,Z_{s})\,{\rm d}s-\int_t^TZ(s)\,{\rm d}W(s),\quad 0\leq t\leq T,
\end{eqnarray}
where for the process solution $(Y,Z)$, the process $(Y_s,Z_{s}) = (Y(s+u),Z(s+u) )_{-T \leq u\leq 0} $ is at each time $s$, its all past values . In order to extend the solution to interval $[-T,0] $ we always suppose that $Y(s)=Y(0)$ and $Z(s)=0$ for $s<0$.

Now, we make the following assumptions on the data $(\xi,f)$.
\begin{description}
\item [({\bf H1})] $\xi$ is $\mathcal{F}_T$-mesurable and integrable variable.
\item [({\bf H2})] $f:\Omega\times [0,T]\times L_{-T}^{\infty}(\mathbb{R}^{k})\times L_{-T}^{2}(\mathbb{R}^{k\times d})\rightarrow \mathbb{R}^{k}$ be a progressively measurable function such   
\begin{enumerate}
\item  for some probability $\alpha$ defined on \newline $([-T,0],\mathcal{B}([-T,0]))$,
\begin{itemize}
\item [$(i)$] there exists a positive constant $K$ satisfying 
\begin{eqnarray*}
&&\mid f(s,y_s,z_{s}) - f(s,y'_s ,z'_{s} )\mid\\
&\leq &  K\int_{-T}^0 ( \mid y(s+u) - y'(s+u) \mid+ \mid z(s+u) - z'(s+u) \mid)\alpha (du),
\end{eqnarray*}
\item [$(ii)$] there  exists $\gamma\geq 0,\; \delta\in (0,1)$ and a non-negative progressively mesurable process $(g(t))_{-T\leq t\leq T}$ such that
\begin{eqnarray*}\label{Assumpbis}
\mid f(s,y_s,z_{s})-f(s,{\bf 0},{\bf 0})\mid 
&\leq &  \gamma\int_{-T}^0 (g(s+u)+\mid y(s+u) \mid+ \mid z(s+u)\mid )^{\delta}\alpha (du),
\end{eqnarray*}
\end{itemize}
for $\P\otimes\lambda $-a.e. $(\omega,s) \in \Omega \times [0,T]$ and for any $(y_s,z_{s}),(y'_s,z'_{s}) \in L_{-T}^{\infty} (\mathbb{R}^k) \times  L_{-T}^2 (\mathbb{R}^{k\times d})$,

\item $\displaystyle\mathbb{E}\left[\int_{0}^{T}(\left|f(t,{\bf 0,0})\right|+g(t))
dt\right]<+\infty$,
\item $\displaystyle f(t,\cdot,\cdot)=g(t)=0$, for $\ t<0$.
\end{enumerate}
\end{description}
Let us end this section by giving the definition of a solution to BSDE \eqref{A} and recall a priori estimates established in \cite{Ral}.

\begin{Definition}
A solution of BSDE \eqref{A} is a pair of a progressively measurable $\R^{k}\times\R^{k\times d}$-valued process $(Y,Z)$ such that: $\P$-a.s., $t\mapsto Z(t)$ belongs to $L^2(0,T)$ and $t\mapsto f(t,Y_t,Z_t)$ belongs to $L^1(0,T)$.
\end{Definition} 
\section{Mains results}
This section is devoted to establish the existence and uniqueness result for BSDE \eqref{A} in $L^1$-sense. For this fact, we shall fist recall a priori estimates established in \cite{ROA}. 
\begin{proposition}\label{lm1}
Assume $\textbf{(H1)}$-$\textbf{(H2)}$ hold.  Let $(Y,Z)$ be a solution to the delayed BSDE \eqref{A} such that for $p>1$, $Y\in \mathcal{S}^{p}$. Then there exist a positive constant $C_{p}$ depending on
 $p,K,T$ such that
\begin{eqnarray*}
\mathbb{E}\left[\left(\int_{0}^{T}|Z(s)|^{2}ds\right)^{\frac{p}{2}}\right]\leq d_p\E(\sup_{0\leq t\leq T}|Y(t)|^{ p})+C_{p}\mathbb{E}\left[|\xi|^p+\left(\int_{0}^{T}|f(s,0,0)|^{2}ds\right)^{\frac{p}{2}}\right],
\end{eqnarray*}
where \\
$
d_p=
\left\{
\begin{array}{llll}
2^{3p/2-2}\left(\frac{1}{2}-2^{p-1}(TK)^{p/2}\right)^{-1}\left(K^{p/2}(T+1)^{p/2}+2^{p/2-1}\lambda^2_p\right)&\mbox{if}\;\; p>2,\\\\
2^{5p/2-2}\left(\frac{1}{2}-2^{2p-2}(TK)^{p/2}\right)^{-1}\left(K^{p/2}(T+1)^{p/2}+2^{p/2-1}\lambda^2_p\right)&\mbox{if}\;\; 1<p<2.
\end{array}
\right.
$\\
 avec\\
 $
\lambda_p= \left\{
\begin{array}{llll}
\left(\frac{p}{p-1}\right)^{p^2/2}\left(\frac{p(p-1)}{2}\right)^{p/2}&\mbox{if}\;\; p>2,\\\\
\left(\frac{4}{p}\right)^{p/4}\frac{4}{4-p}&\mbox{if}\;\; 1<p<2.
\end{array}
\right.
 $
\end{proposition}

\begin{proposition}\label{prop1}
Assume $\textbf{(H1)}$-$\textbf{(H2)}$ hold and $K$ and $T$ small enough.  Let $(Y,Z)$ be a solution to the delayed BSDE \eqref{A}. For any $p>1$, if $Y \in \mathcal{S}^{p}$ then, $Z\in \mathcal{H}^{p}$ and there exist
a positive constant $C_{p}$ depending on $p, K, T$ such that
\begin{eqnarray*}
\mathbb{E}\left[\sup_{0\leq t\leq T}|Y(t)|^{ p}+\left(\int_{0}^{T}|Z(t)|^{2}dt\right)^{\frac{p}{2}}\right]\leq C_{p}\mathbb{E}\left[|\xi|^{p}+\left(\int_{0}^{T}|f(t,0,0)|^{2}dt\right)^{\frac{p}{2}}\right].
\end{eqnarray*}
\end{proposition}
Our mains results are two theorems. The first is a uniqueness result which is stated in Theorem \ref{Th1}. The second is an existence result stated in Theorem \ref{T2}.

\begin{theorem}\label{Th1}
Let assume $({\bf H1})$-$({\bf H2})$ hold and $T$ and $K$ small enough. Then BSDE \eqref{A} has at most one solution $(Y,Z)$ such that $Y$ belongs to $(D)$ and  $Z$ belongs to the space $\displaystyle \bigcup_{\beta>\delta}\mathcal{H}^{\beta}(\R^{k\times d})$.
\end{theorem}
\begin{proof}
Let $(Y,Z)$ and $(Y',Z')$ be two solutions of BSDE \eqref{A} with appropriate condition. For an integer $n\geq 1$, we consider the sequence of stopping time $(\tau_n)_{n\geq 1}$ defined by
\begin{eqnarray}
\tau_{n}=\inf\{t\in [0,T],\; \int^t_0(|Z(s)|^2+|Z'(s)|^2)ds\geq n\}\wedge T.
\end{eqnarray}
Denoting $\Delta Y(t)=Y(t)-Y'(t)$ and $\Delta Z(t)=Z(t)-Z'(t)$, Corollary 3.1 appear in \cite{Bral} yields the inequality
\begin{eqnarray*}
|\Delta Y(t\wedge \tau_n)|&\leq &|\Delta Y(\tau_n)|+\int^{\tau_n}_{t\wedge\tau_n}\langle \widehat{\Delta Y}(s),f(s,Y_s,Z_s)-f(s,Y'_s,Z'_s)\rangle ds\\
&&-\int^{\tau_n}_{t\wedge\tau_n} \langle \widehat{\Delta Y}(s),\Delta Z(s)dW(s)\rangle,
\end{eqnarray*}
where $\hat{x}=|x|^{-1}x{\bf 1}_{\{x\neq 0\}}$.

Thus it follows from Schwarz inequality that 
\begin{eqnarray}\label{Sch}
|\Delta Y(t\wedge \tau_n)|\leq |\Delta Y(\tau_n)|+\int^{\tau_n}_{t\wedge\tau_n}|f(s,Y_s,Z_s)-f(s,Y'_s,Z'_s)|ds-\int^{\tau_n}_{t\wedge\tau_n} \langle \widehat{\Delta Y}(s),\Delta Z(s)dW(s)\rangle.
\end{eqnarray} 
 Hence taking conditional expectation with respect $\mathcal{F}_t$, in both side of \eqref{Sch}, we have
 \begin{eqnarray*}
|\Delta Y(t\wedge \tau_n)|\leq \E\left(|\Delta Y(\tau_n)|+\int^{\tau_n}_{t\wedge\tau_n}|f(s,Y_s,Z_s)-f(s,Y'_s,Z'_s)|ds|\mathcal{F}_t\right).
\end{eqnarray*}
In view of of its definition, it is clear that $\Delta Y$, belongs to $(D)$ and hence it follows that $\P$-a.s, $\Delta Y(\tau_{n})=\Delta Y(\tau_n \wedge T)$ goes to $\Delta Y (T)=0$. Moreover, this convergence hols also in $L^1$. Therefore, it not difficult to prove that the continuous martingale $\E(\Delta Y_{\tau_n}|\mathcal{F}_t)$ leads to $0$ in ucp. According to a subsequence, we get $\P$-a.s that for all $t\in[0,T]$,
\begin{eqnarray}\label{E0}
|\Delta Y(t)|\leq \E\left(\int_0^T|f(s,Y_s,Z_s)-f(s,Y'_s,Z'_s)|ds|\mathcal{F}_t\right),
\end{eqnarray}
which provide in view of $({\bf H2})$-$(ii)$ that
\begin{eqnarray}\label{E1}
&&|\Delta Y(t)|\nonumber\\
&\leq & 2\gamma\E\left(\int_0^T\int^0_{-T}(g(s+u)+|Y(s+u)|+|Y'(s+u)|+|Z(s+u)|+|Z'(s+u)|)^{\delta}\alpha(du)ds|\mathcal{F}_t\right)\nonumber\\
&\leq & 2\gamma\E\left(\int^0_{-T}\int_u^{T+u}(g(s)+|Y(s)|+|Y'(s)|+|Z(s)|+|Z'(s)|)^{\delta}ds\alpha(du)|\mathcal{F}_t\right)\nonumber\\
&\leq & 4\gamma T(|Y(0)|+|Y'(0)|)^{\delta}+2\gamma\E\left(\int_0^{T}(g(s)+|Y(s)|+|Y'(s)|+|Z(s)|+|Z'(s)|)^{\delta}ds|\mathcal{F}_t\right),
\end{eqnarray}
where we use respectively Fubini's equality, change variable and the fact that, for $t<0,\, Y(t)=Y(0),\, Y'(t)=Y'(0)$ and  $Z(t)=Z'(t)=g(t)=0$.
 Next, knowing that there exists $\beta>\delta$ such that $Z$ and $Z'$ belong to $\mathcal{H}^{\beta}$ and since $Y$ and $Y'$ is of class $(D)$, it follows from \eqref{E1} and assumption $({\bf H1})$ that $\Delta Y$ belongs to $\mathcal{S}^{q}$, for some $q>1$. Moreover, using respectively Proposition \ref{lm1} and \ref{prop1}, we derive that $(\Delta Y,\Delta Z)$ is in $\mathcal{S}^q\times\mathcal{H}^q$ and $(\Delta Y,\Delta Z)=(0,0)$.
\end{proof}

Let now give the existence  part of our study. Before giving and prove a general existence result, we need to established the following result which is the case where the generator is
independent of the process $y_t$ and $z_t$, for all $t\in [0,T]$.
\begin{theorem}\label{T3}
Assume \textbf{(H1)}-\textbf{(H2)} and suppose that $f$ does not depend to $y_t$ and $z_t$ for all $t\in [0,T]$. Then BSDE  \eqref{A} has as a unique solution $(Y,Z)$ such that $Y$ belongs to $(D)$. Moreover, for each $\beta\in (0,1),\; (Y,Z) \in \mathcal{S}^{\beta}\times\mathcal{H}^{\beta}$.
\end{theorem}
\begin{proof}
Let $(q_n)_{n\geq 1}$ be sequence of function defined by
\begin{eqnarray*}
q_{n}(x)=\frac{n}{|x|\vee n}x, \;\;\;\;\; \forall\;\; x\in \R. 
\end{eqnarray*}
For each $n\geq 1$, we set 
\begin{eqnarray*}
\xi_n= q_{n}(\xi),\;\;\;\; f_n(t)=q_n(f(t)). 
\end{eqnarray*}
For each $n\geq 1$, $|q_{n}(x)|\leq n$ and hence $\xi_n$ and $f_n$ are respectively bounded and square integrable random. Therefore, it follows from Lemma 2.1 in \cite{PardPeng} that the BSDE
associated to parameter $(\xi_n,f_n$ has a unique solution $(Y^n,Z^n)$ in the space $\mathcal{S}^2\times \mathcal{H}^2$. Setting $\Delta \xi^{n,i}=\xi_{n+i}-\xi_n,\;\Delta Y^{n,i}=Y^{n+i}-Y^n$ and $\Delta Z^{n,i}=Z^{n+i}-Z^{n}$ and applying the same computation as in the uniqueness part (see \eqref{E0}), we have
\begin{eqnarray*}
|\Delta Y^{n,i}(t)|\leq \E\left(|\Delta \xi^{n,i}|+\int_0^T|f_{n+1}(s)-f_{n}(s)|ds|\mathcal{F}_t\right).
\end{eqnarray*}
Thus in view of $q_n$ definition's, we obtain
\begin{eqnarray}\label{E2}
|\Delta Y^{n,i}(t)|\leq \E\left(|\xi|{\bf 1}_{\{|\xi>n\}}+\int_0^T|f(s)|{\bf 1}_{\{|\xi>n\}}ds|\mathcal{F}_t\right),
\end{eqnarray}
from which we deduce without difficulty
\begin{eqnarray}\label{E4}
\|\Delta Y^{n,i}\|_1\leq \E\left(|\xi|{\bf 1}_{\{|\xi>n\}}+\int_0^T|f(s)|{\bf 1}_{\{|\xi>n\}}ds\right).
\end{eqnarray}
Moreover, it follows from Lemma 6.1 in \cite{Bral} that, for any $\beta\in(0,1)$,
\begin{eqnarray}\label{E3}
\E\left(\sup_{t\in [0,T]}|\Delta Y^{n,i}(t)|^{\beta}\right)\leq \frac{1}{1-\beta}\E\left[\left(|\xi|{\bf 1}_{\{|\xi>n\}}+\int_0^T|f(s)|{\bf 1}_{\{|\xi>n\}}ds\right)^{\beta}\right].
\end{eqnarray}
The inequalities \eqref{E3} and \eqref{E4} imply that $Y^{n}$ is a Cauchy sequence for the norm $\|.\|_1$ and for a distance in $\mathcal{S}^{\beta}$, for $\beta\in (0,1)$, define previously in Section 2. Hence $Y$ the progressive measurable continuous process, limit of this sequence, belongs to $(D)$ and $\mathcal{S}^{\beta}$, for $\beta\in (0,1)$.

On the other hand, let recall that $(\Delta Y^{n,i},\Delta Z^{n,i})$ solve this BSDE
\begin{eqnarray*}
\Delta Y^{n,i}(t)=\Delta\xi^{n,i}+\int_t^TF^{n,i}(s)ds-\int^T_t\Delta Z^{n,i}(s)dW(s),
\end{eqnarray*}
where $F^{n,i}(t)=f^{n+i}-f^{n}(t)$.

Therefore, using the Proposition 3.1, we get that for $\beta\in (0,1)$, 
\begin{eqnarray*}
\E\left[\left(\int^T_0|\Delta Z^{n,i}(s)|^2ds\right)^{\beta/2}\right]\leq C_{\beta}\E\left[\left(\sup_{t\in [0,T]}|\Delta Y^{n,i}(t)|^{\beta}+\int_0^T|f(s)|{\bf 1}_{\{|\xi>n\}}ds\right)^{\beta}\right],
\end{eqnarray*}
which provides that $(Z^n)_{n\geq 1}$ is a Cauchy sequence in $\mathcal{H}^{\beta}$, for $\beta\in (0,1)$ for a metric defined in Section 2 and then converge in this space to a progressively measurable process $Z$. Since $\displaystyle \int^t_0 Z^{n}(s)dW(s)$ converge to $\displaystyle \int_0^t Z(s)dW(s)$ in ucp, we provide that by taking the limit ucp in BSDE associated to the data $(\xi^n,f_n)$, the process $(Y,Z)$ solve BSDE with data $(\xi,f)$.  
     
\end{proof}

\begin{theorem}\label{T2}
Assume \textbf{(H1)}-\textbf{(H2)}. If $K$ and $T$ small enough, then BSDE  \eqref{A} has as a unique solution $(Y,Z)$ such that $Y$ belongs to $(D)$. Moreover, for each $\beta\in (0,1),\; (Y,Z) \in \mathcal{S}^{\beta}\times\mathcal{H}^{\beta}$.
\end{theorem}

\begin{proof}
Our method is based on the well-know Picard iteration procedure. For that, let us set $Y^0,Z^0)=(0,0)$ and define recursively, according to Theorem \ref{T3}, for each $n\geq 1$,
\begin{eqnarray}\label{E6}
Y^{n+1}(t)=\xi+\int^T_tf(s,Y^n_s,Z^n_s)ds-\int^T_tZ^{n+1}(s)dW(s),\;\; 0\leq t\leq T,
\end{eqnarray}
where for each $n\geq 1,\; Y^n$ belongs to $(D)$ and $(Y^n,Z^n)$ is in $\mathcal{S}^{\beta}\times\mathcal{H}^{\beta}$, for $\beta\in (0,1)$.

For each $n\geq 1$, it follows from the same computation as in the proof of Theorem \ref{Th1}, that for all $t\in[0,T]$,
\begin{eqnarray}\label{E5}
&&|Y^{n+i}(t)-Y^{n}(t)|
\nonumber\\ &\leq& 4\gamma T(|Y^n(0)|+|Y^{n-1}(0)|)^{\delta}\nonumber\\
&&+2\gamma\E\left(\int_0^{T}(g(s)+|Y^n(s)|+|Y^{n-1}(s)|+|Z^n(s)|+|Z^{n-1}(s)|)^{\delta}ds|\mathcal{F}_t\right).
\end{eqnarray}
Since the processes $Z^{n}$ and $Z^{n-1}$ belong to $\mathcal{H}^{\beta},\, Y^n$ and $Y^{n-1}$ belong to $(D)$, and $(g(t))_{0\leq t\leq T}$ is integrable, we deduce that the random variable
\begin{eqnarray*}
A_n=(|Y^n(0)|+|Y^{n-1}(0)|)^{\delta}+\int_0^{T}(g(s)+|Y^n(s)|+|Y^{n-1}(s)|+|Z^n(s)|+|Z^{n-1}(s)|)^{\delta}ds
\end{eqnarray*}
is $q$-integrable as soon as $\delta q<1$. Next, fix $q$ such that $\delta q<1$, we provide, in view of \eqref{E5}, that $\bar{Y}^{n}=Y^{n+1}-Y^n$ belongs to $\mathcal{S}^{\beta}$. Moreover, setting $\bar{Z}^{n}=Z^{n+1}-Z^{n}$ and $f_n(r,y_r,z_r)=f(r,y_r+Y^{n-1}_r,z_r+Z^{n-1}_r)-f(r,Y^{n-1}_r,Z^{n-1}_r)$, the process $(\bar{Y}^n,\bar{Z}^n)$ satisfy the following BSDE 
\begin{eqnarray*}
\bar{Y}^n(t)=\int^T_tf_n(r,\bar{Y}^n_r,\bar{Z}^n_r)dr-\int^T_t\bar{Z}^n(r)dW(r).
\end{eqnarray*}
Using again corollary 3.1 of \cite{Bral}, we get
\begin{eqnarray*}
|\bar{Y}^n(t)(t)|&\leq &\int^{T}_{t}\langle \widehat{\bar{Y}^n}(s),f_n(s,\bar{Y}^n_s,\bar{Z}^n_s)\rangle ds-\int^{T}_{t} \langle \widehat{\bar{Y}^n}(s),\bar{Z}^n(s)dW(s)\rangle.
\end{eqnarray*}
But $({\bf H2})$-$(ii)$ implies
\begin{eqnarray*}
&&\langle \widehat{\bar{Y}^n}(s),f_n(s,\bar{Y}^n_s,\bar{Z}^n_s)\rangle\\
&\leq & |f(s,Y^n_s,Z_s^n)-f(s,Y^{n-1}_s,Z^{n-1}_s)|\\
&\leq & 2\gamma \int_{-T}^{ø}(g(s+u)+|Y^n(s+u)|+|Y^{n-1}(s+u)|+|Z^n(s+u)|+|Z^{n-1}(s+u)|)^{\delta}\alpha(du).
\end{eqnarray*}
Thus, Proposition 3.1 yields that $\bar{Z}^n$ belongs to $\mathcal{H}^{q}$ since $A_n$ is $q$-integrable.

On the other with Proposition 3.2, we derive
\begin{eqnarray*}
\E\left[\sup_{0\leq t\leq T}|\bar{Y}^n(t)|^q+\left(\int^T_0|\bar{Z}^n(s)|^2ds\right)^{q/2}\right]\leq C_q\E\left[\left(\int^T_0|f(s,Y^n_s,Z_s^n)-f(s,Y^{n-1}_s,Z^{n-1}_s)|ds\right)^q\right],
\end{eqnarray*}
where $C_q$ i a constant depending only on $q$.

For $n\geq 2$, knowing that $f$ is $K$ -delayed Lipschitz, we obtain
\begin{eqnarray*}
\E\left[\sup_{0\leq t\leq T}|\bar{Y}^n(t)|^q+\left(\int^T_0|\bar{Z}^n(s)|^2ds\right)^{q/2}\right]\leq C\E\left[\sup_{0\leq t\leq T}|\bar{Y}^{n-1}(t)|^p+\left(\int^T_0|\bar{Z}^{n-1}(s)|^{2}ds\right)^{q/2}\right],
\end{eqnarray*}
where $C=C_qK^{p/2}\max(1,T^{p/2})$. Thus, recursively we get 
\begin{eqnarray*}
\E\left[\sup_{0\leq t\leq T}|\bar{Y}^n(t)|^q+\left(\int^T_0|\bar{Z}^n(s)|^2ds\right)^{q/2}\right]\leq C^{n-1}\E\left[\sup_{0\leq t\leq T}|\bar{Y}^{1}(t)|^p+\left(\int^T_0|\bar{Z}^{1}(s)|^{2}ds\right)^{q/2}\right],
\end{eqnarray*}
 Since $(Y^1,Z^1)$ belongs to $\mathcal{S}^{q}\times\mathcal{H}^q$ and $C_qK^p\max(T^{p/2},1)<1$, it follows that $(Y^n,Z^n)$ is a Cauchy sequence on $\mathcal{S}^{q}\times\mathcal{H}^q$ and then converge in $\mathcal{S}^{q}\times\mathcal{H}^q$ to the process $(Y,Z)$. Better this convergence takes place in $\mathcal{S}^{\beta}\times\mathcal{H}^{\beta}$ for $\beta\in (0,1)$. Moreover, $Y^n$ converge to $Y$ for the norm $\|.\|_1$ since the convergence of $\mathcal{S}^q$ with $q\geq 1$ is stronger than the $\|.\|_1$-norm.
 
 To end the proof, it remains to pass to the limit in \eqref{E6} to prove that $(Y,Z)$ solve BSDE \eqref{A}.
\end{proof}

\label{lastpage-01}
\end{document}